\documentclass[12pt]{article}
\usepackage{amssymb,epsfig,amsfonts,epic,graphics,pictex}
\begin{document}

\newtheorem{lem}{Lemma}[section]
\newtheorem{prop}{Proposition}
\newtheorem{con}{Construction}[section]
\newtheorem{defi}{Definition}[section]
\newtheorem{coro}{Corollary}[section]
\newcommand{\hf}{\hat{f}}
\newtheorem{fact}{Fact}[section]
\newtheorem{theo}{Theorem}
\newcommand{\Br}{\Poin}
\newcommand{\Cr}{{\bf Cr}}
\newcommand{\dist}{{\rm dist}}
\newcommand{\diam}{\mbox{diam}\, }
\newcommand{\mod}{{\rm mod}\,}
\newcommand{\compose}{\circ}
\newcommand{\dbar}{\bar{\partial}}
\newcommand{\Def}[1]{{{\em #1}}}
\newcommand{\dx}[1]{\frac{\partial #1}{\partial x}}
\newcommand{\dy}[1]{\frac{\partial #1}{\partial y}}
\newcommand{\Res}[2]{{#1}\raisebox{-.4ex}{$\left|\,_{#2}\right.$}}
\newcommand{\sgn}{{\rm sgn}}

\newcommand{\CC}{\mathbb{C}}
\newcommand{\D}{{\bf D}}
\newcommand{\Dm}{{\bf D_-}}
\newcommand{\RR}{\mathbb{R}}
\newcommand{\NN}{\mathbb{N}}
\newcommand{\HH}{\mathbb{H}}
\newcommand{\ZZ}{\mathbb {Z}}
\newcommand{\tr}{\mbox{Tr}\,}
\newcommand{\R}{{\bf R}}
\newcommand{\C}{{\bf C}}

\newenvironment{nproof}[1]{\trivlist\item[\hskip \labelsep{\bf Proof{#1}.}]}
{\begin{flushright} $\square$\end{flushright}\endtrivlist}
\newenvironment{proof}{\begin{nproof}{}}{\end{nproof}}

\newenvironment{block}[1]{\trivlist\item[\hskip \labelsep{{#1}.}]}{\endtrivlist}
\newenvironment{definition}{\begin{block}{\bf Definition}}{\end{block}}

\newtheorem{conjec}{Conjecture}

\newtheorem{com}{Comment}
\font\mathfonta=msam10 at 11pt
\font\mathfontb=msbm10 at 11pt
\def\Bbb#1{\mbox{\mathfontb #1}}
\def\lesssim{\mbox{\mathfonta.}}
\def\suppset{\mbox{\mathfonta{c}}}
\def\subbset{\mbox{\mathfonta{b}}}
\def\grtsim{\mbox{\mathfonta\&}}
\def\gtrsim{\mbox{\mathfonta\&}}

\newcommand{\Poin}{{\bf Poin}}
\newcommand{\Bo}{\Box^{n}_{i}}
\newcommand{\Di}{{\cal D}}
\newcommand{\gd}{{\underline \gamma}}
\newcommand{\gu}{{\underline g }}
\newcommand{\ce}{\mbox{III}}
\newcommand{\be}{\mbox{II}}
\newcommand{\F}{\cal{F}}
\newcommand{\Ci}{\bf{C}}
\newcommand{\ai}{\mbox{I}}
\newcommand{\dupap}{\partial^{+}}
\newcommand{\dm}{\partial^{-}}
\newenvironment{note}{\begin{sc}{\bf Note}}{\end{sc}}
\newenvironment{notes}{\begin{sc}{\bf Notes}\ \par\begin{enumerate}}%
{\end{enumerate}\end{sc}}
\newenvironment{sol}
{{\bf Solution:}\newline}{\begin{flushright}
{\bf QED}\end{flushright}}

\title{Measure of the Julia Set of the Feigenbaum map
with infinite criticality}

\author{Genadi Levin
\thanks{Both authors were supported by Grant No. 2002062 
from the United States-Israel Binational Science Foundation
(BSF), Jerusalem, Israel.}\\
\small{Dept.\ of Math.}\\
\small{Hebrew University}\\
\small{Jerusalem 91904, Israel}\\
\small{\tt levin@math.huji.ac.il}\\
\and
Grzegorz \'{S}wia\c\negthinspace tek\\
\small{Wydzia\l\ MiNI}\\
\small{Politechnika Warszawska}\\
\small{Plac Politechniki 1}\\
\small{00-661 Warszawa, Poland}\\
\small{\tt g.swiatek@mini.pw.edu.pl}}
\normalsize
\maketitle
\abstract{
We consider fixed points of the Feigenbaum (periodic-doubling)
operator whose orders tend to infinity.
It is known that the hyperbolic
dimension of their Julia sets go to $2$. We prove that
the Lebesgue measure of these Julia sets tend to zero.
An important part of the proof  consists in applying martingale
theory to a stochastic process with non-integrable increments. 
}

\section{Introduction}


We consider fixed points of the Feigenbaum (periodic-doubling)
operator~\cite{feig2} whose orders tend to infinity.
It has been shown in~\cite{LSw2},~\cite{LSw3}, that the hyperbolic
dimension of their Julia sets go to $2$. In this paper we prove that
the Lebesgue measure (area) of these Julia sets tend to zero.
The question whether the area is indeed zero for finite orders
remains open. For the measure problem for maps with Fibonacci combinatorics,
see~\cite{SN},
and for quadratic Julia sets with positive area, see~\cite{BC}.







\paragraph{Outline of the proof.}
We follow the path known as ``the random walk argument''. 
In Sect. 3 we build a Markov partition
by modifying the partition that we used in~\cite{LSw2}.
This partition defines a ``level function'' on the phase space 
which tends to $+\infty$ at $\infty$ and to $-\infty$ at $0$.  
With respect to the level function, the dynamics of the tower of the
limit map defines a random process. We then study the 
probability distribution for this process and finally show that for
almost every point the process oscillates between $-\infty$ and
$+\infty$. The last step uses a martingale argument in the spirit 
of~\cite{BKNS}.

The process we study has transition probabilities that are
asymptotically symmetric with respect to the change of the sign and
their magnitude is $\sim n^{-2}$. This, of course, makes them
non-integrable. There has been a considerable interest in such process
coming from probability theory. The simplest case is a Markov process
$X_n$ 
with independent increments with a distribution law of this type. That
case was studied in~\cite{Linnik} with a further development
in~\cite{Russian, Aronson1}. The main result is that $\frac{X_n}{n} - c$ tends
in probability to an analytic limit distribution law. From this, it is
easy to conclude that almost every orbit oscillates between $-\infty$
and $+\infty$. This was then extended by~\cite{Aronson2} to the
dynamical context of iterated function systems with distortion, that
is, the case in which the increments are no longer independent. The
result about a limit distribution law, under suitable assumptions,
remains the same.

One important corollary from our results, which was announced
previously~\cite{LSw3}, p.424, is that the Julia set of the limit
transcendental map has area $0$, see Theorem~\ref{A}. 

\paragraph{Acknowledgment.} The authors thank 
Yuval Peres, Benjamin Weiss,
Jon Aaronson, and Michel Zinsmeister for valuable suggestions
and discussions on different stages of present work.

\subsection{Main results}

\paragraph{Notations and basic facts.}
We will write unimodal mappings of an interval, 
$H :\: [0,1] \rightarrow [0,1]$ in the following
non-standard form:

\[ H(x) = |E(x)|^{\ell} \]

where $\ell>1$ is a real number and 
$E$ is an analytic mapping with strictly negative
derivative on $[0,1]$ which maps $0$ to $1$ and $1$ to a point inside
$(-1,0)$. Then $H$ is unimodal with the minimum at some 
$x_{0} = E^{-1}(0) \in (0,1)$ and $x_0$ is the critical
point of order $\ell$. 

For $\ell$ which are
even integers there exists a unique pair $H:=H^{(\ell)}(x) =
|E_{\ell}(x)|^{\ell}$ and $\tau:=\tau_{\ell}>1$ which provides a
solution to the {\em Feigenbaum functional equation}
 
\begin{equation}\label{equ:14fa,1}
\tau H^2\tau^{-1}(x) = H(x) 
\end{equation}
for $x \in [0,1]$.

As $\ell$ goes to $\infty$, mappings $H^{(\ell)}$ converge to a
non-trivial analytic limit denoted by $H$~\cite{oldwit, LSw1}.
It satisfies the Feigenbaum equation with $\tau=\lim\tau_\ell>1$. 
According to ~\cite{LSw1}, the limit map $H$ extends
to an infinite unbranched cover
of either of two topological disc $U_{-}$ and $U_{+}$
onto a punctured round disc
$D_*=D(0, R)\setminus \{0\}$. Here
$U=U_{-}\cup U_+$ is compactly contained in the disc
$D(0, R)$ and $U_{\pm}$ touch each other at a single
point $x_0$, which is the limit
of the critical points for $H^{(\ell)}$.
In particular, the (filled) Julia set
$J(H)$ of $H$ is well-defined as the closure
of non-escaping points of $H: U\to D_*$.
$J(H)$ has no interior.
 

\paragraph{Statements.}

\begin{theo}\label{A}
The Julia set $J(H)$ of $H$ has area zero.
\end{theo}

Note that by~\cite{LSw2},~\cite{LSw3} the hyperbolic
(in particular, Hausdorff) dimension of $J(H)$ is $2$.

A stronger result is presented in Theorem~\ref{rec},
which provides an additional property of the
corresponding tower dynamics, roughly that almost every point visits
every neighborhood of both $0$ and $\infty$.

\begin{coro}\label{finite} 
The area of the Julia set of the map
$H^{(\ell)}$ tends to zero
as the order $\ell$ grows.
\end{coro}


Theorem~\ref{A} together with Theorem 7 of~\cite{LSw1} and 
~\cite{LSw2}
immediately imply

\begin{coro}\label{example} 
There exist real parameters $a,c>0$, such that the map
\[f(z)=a\exp{(-(z-c)^{-2})}\] 
has the following properties:

(a) $f$ is quasi-conformally conjugate to $H$ on the entire domain of $H$,

(b) the set of points in the plane whose $\omega$-limit sets under $f$
are contained in the $\omega$-limit set of $0$ has Hausdorff dimension $2$,

(c) the hyperbolic dimension of the Julia set $J(f)$
of $f$ is equal to $2$,

(d) the area of $J(f)$ is equal to zero.
\end{coro}


\section{Induced Dynamics}
We will build on~\cite{LSw2} adopting the notations of that paper. 

\subsection{Limit Feigenbaum map}

The following statement proved in~\cite{LSw1},~\cite{LSw2} describes a maximal
dynamical extension of the map $H: U\to D_*$ and related facts.

\begin{theo}\label{prop}

(1) On the interval $[0,1]$, $H^{(\ell)}$ converge uniformly to a unimodal
map $H$ with a critical point at $x_0$ which satisfies the Feigenbaum
fixed point equation~(\ref{equ:14fa,1}) with some $\tau > 1$. 

(2) There is an analytic map $h$ defined on the union of two open 
topological disks
$\Omega_- \ni 0$ and $\Omega_+$, both symmetric with respect to the
real axis with closures intersecting exactly at $x_0$. 

(3) $\Omega_+$ and $\Omega_-$ are bounded and their boundaries 
are Jordan curves with
Hausdorff dimension $1$. 
Moreover, $\overline\Omega_{\pm}\cap \RR=\overline{\Omega_{\pm}\cap \RR}$.

(4) $h$ is
univalent on $\Omega_-$ and maps it onto $C_h := \CC \setminus \{ x\in \RR
:\: x \geq 2\log\tau\}$ and also univalent on $\Omega_+$ mapping it
onto $\CC \setminus \{ x \in \RR :\: x \geq\log\tau\}$. 

(5) On any compact subset of $\Omega_+ \cup \Omega_-$, $H^{(\ell)}$ are
defined and analytic for all $\ell$ large enough and converge
uniformly to $H:=\exp(h)$, which is an analytic extension
of the map $H: U\to D_*$ previously introduced. 

(6) if $G=H\circ \tau^{-1}$, then $h\circ G=h-\log \tau$ on $\Omega_{\pm}$.
That is to say, the map $h/(-2\log \tau)$ is an attracting
Fatou coordinate for $G^2$ at $x_0$.
$G^{-1}(\Omega_+) = \Omega_-$ and $G^{-1}(\Omega_- \setminus [y,0]) =
\Omega_+$ where $G^{-1}$ 
is an inverse branch of $G$ defined on $\CC
\setminus \left( (-\infty,0] \cup [\tau x_0, + \infty)\right)$ which fixes
$x_0$ and $y<0$ is chosen so that
$G^2$ maps $(y, x_0)$ monotonically onto $(0, x_0)$. 

(7) $\Omega_+ \subset \tau \Omega_-$ and $\Omega_- \subset \tau \Omega_-$

\end{theo}

\begin{figure}\label{fig:16fa,1}
\psfig{figure=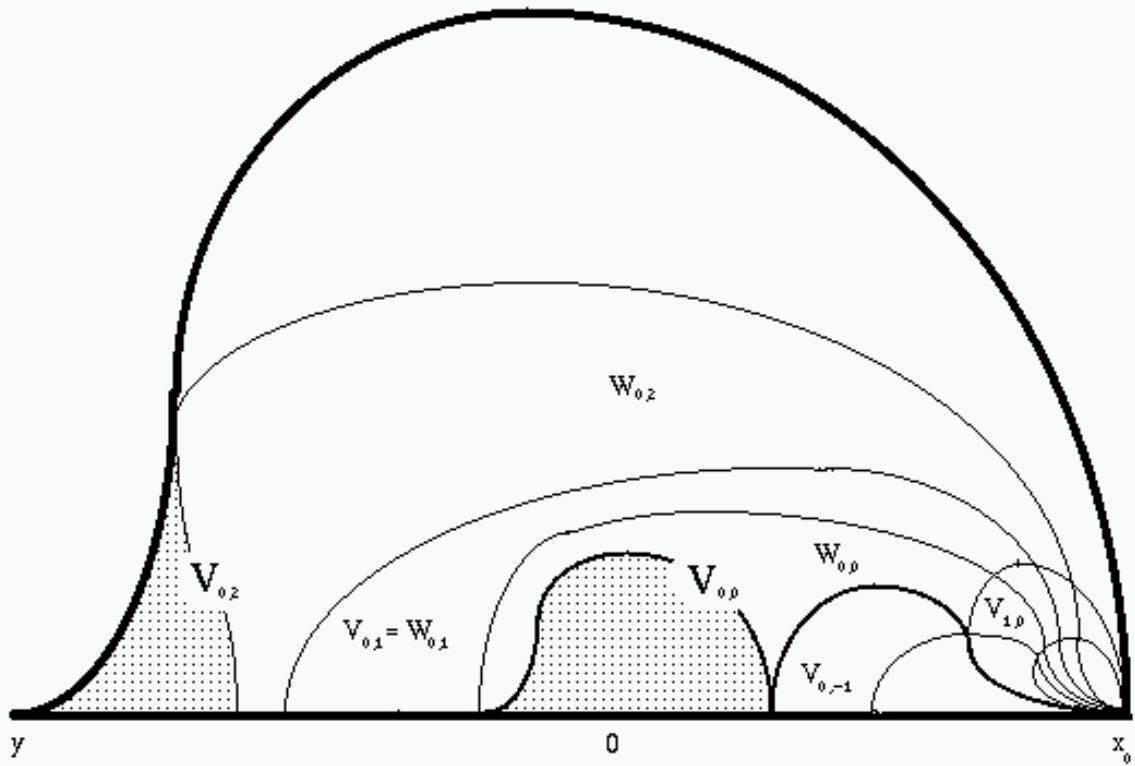,width=528pt,height=396pt}
\caption{This is a schematic drawing of $\Omega_- \cap \HH_+$ and regions
inside it. Areas delineated with thicker lines represent
$\tau^{-1}\Omega_-$ and $\tau^{-1}\Omega_+$. Shaded areas correspond
under $G^2$.}
\end{figure}

\paragraph{The geometry of $H$.}
See Figure~\ref{fig:16fa,1} for an illustration and explanation of
some notations. 

Let us define $B = \Omega_- \setminus \tau^{-1}\overline{\Omega_-}$
and $B_{\pm} := B \cap \HH_{\pm}$. Then define $D_{\pm} = B_{\pm} \setminus
\tau^{-1}\overline{\Omega_+}$. 

A convenient parametrization of the set $\Omega_-$ is given 
by the map $h^{-1}$
from a slit plane $C_h$, as described by item~(4) of Theorem~\ref{prop}. If
we write $w=h(z)$, then the map $H(z)$ corresponds to $\exp(w)$ and,
more strikingly $z\rightarrow G^2(z)$ is conjugated to $w\rightarrow w -
2\log \tau$. Geometrically, it is worth noting that the beginning of
the slit at $w=2\log\tau$ corresponds to the point $y$ in
Figure~\ref{fig:16fa,1} where the boundaries of $\Omega_+$ and
$\Omega_-$ which follow the real line to the right of $y$, split.  
 
Following~\cite{LSw2}, connected sets $V_{k,k'}$, $k,k'\in \ZZ$,  
are chosen in $\Omega_-$ so that 
each is mapped by $H$  
onto $\tau^{k'+1}B_{\pm}$.

More explicitly, in the
$w$-coordinate 
\[ h(V_{k,k'}) = \{ w\in C_h :\: \exp(w-k'\log\tau) \in \tau B_{\pm},\, k\pi < \Im
w < (k+1)\pi \} \; .\]

Now $\tau^{-1}\Omega_- \subset \overline{V_{0,0} \cup V_{-1,0}}$. 
Hence, for $k=0,-1$
and $k'$ even and non-negative, $V_{k,k'}$ contains the preimage of
$\tau^{-1}\Omega_{-}$ by $G^{k'}$. To exclude this preimage, if $k=0$ or $k=-1$
and $k'$ is even end non-negative, we define  $W_{k,k'} = V_{k,k'} \setminus
G^{-k'}(\tau^{-1}\Omega_-)$. For all other pairs $(k,k')$, $W_{k,k'} = V_{k,k'}$.

\paragraph{Rescaled map.}
Let us define the ``rescaled map'' $\tilde{H}$ as follows. 

\begin{itemize}
\item
if $z\in W_{k,k'}$, $(k,k') \in \ZZ \times \ZZ$, then $\tilde{H}(z) =
\tau^{-k'-1}H(z)$, 
\item
if $z\in V_{k,k'} \setminus W_{k,k'}$, $(k,k') \in \{0,-1\} \times
2\ZZ_+$, 


then first consider $Z = G^{k'}(z)$. $Z\in
\overline{\tau^{-1}\Omega_-}$ which consists of the rescaled copies of
$B_{\pm}$. Then $\tilde{H}(z)$ is defined if and only if
$Z\in\tau^{-p}B_{\pm}$ for some $p>0$ and then $\tilde{H}(z)=\tau^p Z$.   
\end{itemize}

Notice that this definition ensures that $\tilde{H}$ maps every
connected component of its domain univalently onto one of the four possible
pieces $D_{\pm}, B_{\pm}$. The image is $B_{\pm}$ in the
second case of the definition of the rescaled map and also in the
first case whenever $W_{k,k'}=V_{k,k'}$. 

On the other hand $\tilde{H}$ on $W_{0,0}$ is 
$\tau^{-1} H = \tau^{-1} G \tau$. It maps $W_{0,0}$
univalently onto $D_-$ since $G$ maps $\Omega_-$ onto $\Omega_+$. 
On $W_{-1,0}$,
$\tilde{H}$ is the mirror reflection of this map, so the same formula
actually holds. We will refer to
$W_{0,0}, W_{-1,0}$ as the {\em central pieces}. 

On $W_{2p,0}$, $p>0$, $\tilde{H}=\tilde{H}_{|W_{0,0}} \circ G^{2p}$,
so it also maps onto $D_-$ and similarly $W_{-1,2p}$ is mapped onto
$D_+$.  

Distortion properties of $\tilde{H}$ are given by Proposition 1
of~\cite{LSw2}. As it turns out, most branches of $\tilde{H}^n$ can be
be continued as univalent maps onto fixed neighborhoods of $B_{\pm},
D_{\pm}$, fixed meaning independent of a branch or $n$, with the
exception of those branches whose domains are send to the central
pieces by $\tilde{H}^{n-1}$. 

\paragraph{Towers.}
The following is a trivial application of the concept of a tower
used in~\cite{profesorus}.  

\begin{defi}\label{defi:20ga,1}
Suppose we have a pair $(H,\tau)$ which satisfies the
equation~\ref{equ:14fa,1}.  For every $k\in \ZZ$, $H$ gives rise to a
rescaled mapping $H_k(z) = \tau^k H (\tau^{-k} z)$. The set $\{ H_k
:\: k\in\ZZ \}$ will be called  the {\em tower of $H$}.
The set of all possible compositions of maps from a tower will be
referred to as {\em tower dynamics}. 
\end{defi}

Towers will be used when $H$ could be the limiting map discussed in
the previous paragraph, or one of the fixed point transformations $H^{(\ell)}$ of
finite degree.

Tower dynamics forms a
dynamical system, namely it defines an action of the semi-group of 
non-negative binary
rational numbers under which integers correspond to ordinary iterates
of $H$ and $2^{-k}$ acts as $H_k$. This follows from the following
lemma.

\begin{lem}\label{lem:20ga,1}
For every $k\in\ZZ$, $H_k^2 = H_{k-1}$. 
\end{lem}
\begin{proof}
Based on the functional equation~\ref{equ:14fa,1}, 
\[ H_k \circ H_k(x) = \tau^{k-1} \tau H^2 (\tau^{-1}(\tau^{-k+1}x)) = 
\tau^{k-1} H(\tau^{-k+1}x) = H_{k-1}(x) \; .\]
\end{proof}

\subsection{Further inducing}
Map $\tilde{H}$ has satisfactory properties from the
combinatorial point view, since $B_+$ and $B_-$ are cut into countably
many topological disks, each is of which is mapped univalently back
onto $B_+$ or $B_-$. However, we would like it to have bounded
distortion and that is generally not so. A standard approach to
obtaining bounded distortion is by inducing and will follow that route
now.

Start by introducing new pieces $K_{\pm} = B_{\pm}\setminus
\overline{W_{0,0}\cup W_{-1,0}}$ and $L_{\pm} = D_{\pm} \setminus
\overline{W_{0,0} \cup W_{-1,0}}$. 

We will next define the map $\tilde{C}$ almost everywhere on the union
of the central pieces, induced by $\tilde{H} = \tau^{-1} G \tau$,
for which every branch maps on $L_{\pm}$. This will allow us to build
the map $\tilde{J}$ defined on the domain of $\tilde{H}$ by
$\tilde{H}$ except on $\tilde{H}^{-1}(W_{0,0}\cup W_{-1,0})$ and by
$\tilde{C} \circ \tilde{H}$ otherwise.

\paragraph{The mapping $\tilde{C}$.}
We will only consider $\tilde{C}$ on $W_{0,0}$. Mapping on $W_{-1,0}$
will be the mirror reflection. 

By Lemma 2.13 in~\cite{LSw2}, no point will stay forever in $W_{0,0}$
under the iteration by $\tilde{H}$. Hence, $\tilde{C}$ is simply
defined as the first entry map into $L_{+}$ under the iteration by
$\tilde{H}$. 

\begin{lem}\label{lem:22xp,1}
Every branch of $\tilde{C}$ has a univalent extension onto a simply
connected neighborhood $U_L$ of $L_{+}$. $U$ is the same for all
branches of $\tilde{C}$ and its preimages by any branch is contained
in the set $S := \{ z :\: \Re z < 0\; \mbox{or}\; \Im z > 0\}$. 
\end{lem}
\begin{proof}
Since $\tilde{H}$ on $W_{0,0}$ is $G_{-1} = \tau^{-1} G \tau$ 
and $\overline{L_+}$
only intersects $\RR_+$ at $x_0$ which is not a critical
point of $\tilde{H}$, we can define and inverse branch on a
neighborhood of $L_+$. Since the preimage of $\overline{L_+}$ by
$G_{-1}$ now only
intersects the real line at $y$, that neighborhood $U_L$ can be chosen to
fit into $S$. This proved the needed extension for the branch of
$\tilde{C}$ which is the first iterate of $\tilde{H}$. To examine
further branches, continue mapping by the inverse branch of
$\tilde{H}$ defined on $S$. From the properties of $G$, that inverse
branch  sends $S$ into itself, or even into the upper half plane.  
\end{proof}

\paragraph{Bounded distortion for $\tilde{J}$.}
Now define map $\tilde{J}$ which equals $\tilde{H}$
everywhere on the domain of $\tilde{H}$ except on preimages to the central pieces
and $\tilde{C}\circ \tilde{H}$ on such preimages. 

Each branch of $\tilde{J}$ maps onto one of the pieces $K_{\pm},
L_{\pm}$.  

\begin{prop}\label{prop:22xp,1}
There are fixed neighborhoods of sets $\overline{K_{\pm}}$ and
$\overline{L_{\pm}}$ such that for any $n$ any branch of $\tilde{J}^n$
which maps onto one those sets can also be extended univalently so
that it maps onto the corresponding neighborhood.
\end{prop}

The map $\tilde{J}^n$ can be expended as a composition of $\tilde{H}$
and $\tilde{C}$ in which $\tilde{C}$ cannot be followed by another
$\tilde{C}$. Since $\tilde{C}$ is also induced by $\tilde{H}$, we can
use Proposition 1 from~\cite{LSw2}.  It asserts that if $\tilde{H}$ is
the last mapping applied in this composition, then the claim of
Proposition~\ref{prop:22xp,1} holds. So consider the situation when
$\tilde{C}$ is applied last. By Lemma 2.10 from~\cite{LSw2}, the
entire composition that comes before it can be continued so that it
maps onto $\CC \setminus [0,+\infty)$. This obviously contains the set
$S$ mentioned in Lemma~\ref{lem:22xp,1}, so again we get a univalent
extension mapping over a fixed neighborhood of $\overline{L_{\pm}}$. 

This proves Proposition~\ref{prop:22xp,1}. 

By K\"{o}be's Lemma we now know that all branch of $\tilde{J}^n$ have
distortion bounded uniformly with respect to $n$.

\subsection{Tower Dynamics}
By Lemma 2.14 from~\cite{LSw2}, for every branch $\zeta$ of $\tilde{H}$ there
exists an integer $p$ such that $\tau^p \zeta$ belongs to the tower of
$H$, see Definition~\ref{defi:20ga,1}. Let us call $p$ the {\em combinatorial displacement} of $\zeta$. 

This leads to the following definition. 

\begin{defi}\label{defi:2ja,1}
A mapping defined on an open set contained in the fundamental ring 
$\Omega_- \setminus \tau^{-1}\overline{\Omega_-}$ is called {\em
tower-induced} if on each connected component of its domain it has the
form $\tau^{q} h$ where $q$ is an integer and $h$ belongs to the tower
of $H$. 
\end{defi}

For a tower-induced mapping, the choice of $q$ and $h$ is unique. 
\begin{lem}\label{lem:2ja,1}
If $\tau^q h = \tau^{q'} h'$ on a connected open set, with
$q,q'\in\ZZ$ and $h,h'$ in the tower of $H$, then $q=q'$ and $h=h'$. 
\end{lem}
\begin{proof}
Both $h$ and $h'$ are iterates of the same $h_0$ in the tower, say 
$h=h_0^m, h'=h_0^{m'}$, $m,m' > 0$. Without loss of generality, $m'
\geq m$. Then 
\[ \tau^{q-q'} = h_0^{m'-m} \]
on an open set, but this is impossible given that no iterate of $h_0$ is a
linear map. 
\end{proof}

\begin{defi}\label{defi:21ga,1}
Given a tower-induced map $\Phi$  on a subset $U$ of the fundamental
ring, we can define its {\em associated map} as follows. 
On $U$, wherever $\Phi = \tau^q h$, the associated map is just $h$. 
On $\tau^p U$, where $p\in \ZZ$, the associated map is $\tau^p h
\tau^{-p}$. 
\end{defi}

In this way, the associated map belongs to the tower.   

\begin{lem}\label{lem:2ja,2}
If the combinatorial displacement of a tower induced map is $q$ at
some point $x$, then for any $p\in \ZZ$ the associated map sends
$\tau^p x$ in $\tau^{p+q} (\Omega_- \setminus
\tau^{-1}\overline{\Omega_-})$. 
\end{lem}
\begin{proof}
It is a direct consequence of the definitions.
\end{proof}

\begin{lem}\label{lem:22xp,2}
If $\zeta_1$ and $\zeta_2$ are two tower-induced mappings with
associated maps $\Theta_1,\Theta_2$, respectively,  then
$\zeta_1\circ\zeta_2$ is also a tower-induced map with the associated
map $\Theta_1 \circ \Theta_2$. 
\end{lem}
\begin{proof}
Denote $\zeta_1 = \tau^{q_1} h_1, \zeta_2 = \tau^{q_2} h_2$.  Without
loss of generality, the domain of $\zeta_2$ is connected and so $q_2$
is constant, while $q_1$ is only piecewise constant and $h_1$ is only
piecewise a map from the tower. 

Then 
\[ \zeta_1 \circ \zeta_2 = \tau^{q_1} h_1 \tau^{q_2} h_2 =
\tau^{q_1+q_2} (\tau^{-q_2} h_1 \tau^{q_2}) h_2 \; .\]
Mappings $h_2$ and $\tau^{-q_2} h_1 \tau^{q_2}$ both belong to the
tower and so does their composition. Thus, the composition $\zeta_1
\circ \zeta_2$ is tower-induced and its associated map is 
$(\tau^{-q_2} h_1 \tau^{q_2}) h_2$ on the domain of $\zeta_2$. 
On the other hand, $h_2$ maps the domain of $\zeta_2$ into
$\tau^{-q_2} (\Omega_- \setminus \tau^{-1} \Omega_-)$. So, the
composition of the associated maps is indeed 
\[  (\tau^{-q_2} h_1 \tau^{q_2}) h_2 \] 
on the domain of $\zeta_2$. So, the associated map of the composition
is equal to the composition of the associated maps on the domain of
$\zeta_2$. When considered on rescaled images of the domain of
$\zeta_2$, both $\Theta_1 \circ \Theta_2$ and the associated map of
$\zeta_1\circ\zeta_2$ are equivariant with respect to such rescalings,
so the equality holds everywhere. 
\end{proof}

As a consequence of Lemma~\ref{lem:2ja,2} and Lemma~\ref{lem:22xp,2},
combinatorial displacements are additive under the composition of
tower-induced maps. 

\paragraph{Dynamical interpretation of $\tilde{J}$.}
Let us recall the mapping $\tilde{J}$ defined previously. Map
$\Lambda$ is equal to $\tilde J$ except on $\tau^{-1}\Omega_+$, where
we modify the definition to ${\tilde J} \circ {\tilde J}$. 

\begin{prop}\label{prop:2jp,1}
If $x$ belongs to the Julia set of $H$ and to the domain of
$\Lambda^p$, $p>0$, then the map associated to $\Lambda^p$ is equal
to an iterate of $H$ on a neighborhood of  $x$. 
\end{prop}

Throughout this proof we assume that $x$ belongs to the Julia set of
$H$. 

We split the proof depending on whether $x$ belongs to $\Omega_+$ or
$\Omega_-$. 

The first case to consider is $x\in \Omega_+$. To determine the map
associated to $\Lambda$ on a neighborhood of $x$, we need to look at
$\Lambda$ on a neighborhood of $\tau^{-1}x \in \tau^{-1} \Omega_+$. 
By the modification we just described, $\Lambda$ is ${\tilde J}^2$ on
a neighborhood of $\tau^{-1}x$ and so the associated map at
$\tau^{-1}x$, as well as $x$, is the associated map of ${\tilde J}$
composed with itself. 

On $\Omega_+$ the associated map $\tilde J$ is $H_1 = \tau H
\tau^{-1}$. Then we know that $H_1(H_1(x)) = H(x)$ is in
$\Omega_-\cup\Omega_+ \subset \tau\Omega_-$. i.e. $H_1(x) \in H_1^{-1}(\tau\Omega_-) \cap
\tau\Omega_-$ or  $\tau^{-1} H_1(x) \in H^{-1}(\Omega_-) \cap
\Omega_-$. It follows that ${\tilde J}$ on a neighborhood of ${\tilde
J}(\tau^{-1}x) = \tau^{-1}H_1(x)$ is $H$, and therefore its associate 
map at $H_1(x)$ is $H_1$ again. So, by Lemma~\ref{lem:22xp,2}, the
associated map of $\Lambda$ is $H_1 \circ H_1 = H$ in a neighborhood
of $x$.

Let us now consider $x\in\Omega_-$. 

Since ${\tilde J} = {\tilde C} \circ {\tilde H}$, with both $\tilde C$
and $\tilde H$ tower-induced maps, we have an analogous decomposition
of the map $\Phi_J$ associated to $\tilde J$ into the composition of
$\Phi_H$ associated to $\tilde H$ and $\Phi_C$ associated to $\tilde
C$. 

\begin{lem}\label{lem:2jp,1}
Suppose $x\in \Omega_-$. Then $\Phi_C$ on a neighborhood of $x$ is an
iterate of $H_{-1}$.  
\end{lem}
\begin{proof}
$\tilde C$ is induced by the map $G_{-1} = \tau^{-1} H$. So, the
associated map is $H$ on the fundamental ring $\Omega_- \setminus
\tau^{-1}\overline{\Omega_-}$. However, the combinatorial displacement
of $G_{-1}$ is $1$, so by Lemma~\ref{lem:22xp,1} the map associated 
to ${\tilde C}^2$ is $H_1 \circ H$ and, inductively, the map
associated to ${\tilde C}^k$, $k\geq 1$, is $H_{k-1} \circ \cdots
\circ H$ wherever ${\tilde C}^k$ is defined on the fundamental ring.  

Observe that $H$ maps any point in the domain of ${\tilde C}$ outside
of $\Omega_- \cup \Omega_+$. Hence, no $x$ from the Julia set of $H$
can be found there.  However, we may encounter points from the Julia
set on the domain of ${\tilde C}$ rescaled by $\tau^{-p}$, $p>0$. By
the equivariance with respect to the rescaling by $\tau$, the map
associated to ${\tilde C}^k$ on a neighborhood of such a point is
$H_{k-1-p} \circ \cdots \circ H_{-p}$. Again, this composition cannot
contain $H_0 = H$ which would eject the point out of the Julia set,
hence $k-1-p<0$, hence $\Phi_C$ is generated by $H_{-1}$ in the
neighborhood of $x$. 
\end{proof}

In the light of Lemma~\ref{lem:2jp,1} in order to conclude that
$\phi_C \circ \Phi_H$ is an iterate of $H$ in a neighborhood of $x$ it
will be enough to show that $\Phi_H$ is an iterate of $H$ on such a
neighborhood. Then, $\Phi_H(x)$ is in the Julia set of $H$ and
Lemma~\ref{lem:2jp,1} is applicable. 

${\tilde H}$ is simply $\tau^q H$ on most of its domain, with the sole
exception of domains $G^{-2k}(\tau^{-1}\Omega_-)$ where the inverse
branch of $G$ which fixes $x_0$ is used. On any such domain,
$\tilde{H}$ is $\tau^q G^{2k}$. Since $G=\tau^{-1} H_1$ its associated
map is $H_1$ and the combinatorial displacement is $1$. Hence, the map
associated to $\tilde{H}$ on such a domain is 
\begin{equation}\label{equ:2jp,1}
H_{2k} \circ H_{2k-1} \circ \cdots \circ H_1 \; .
\end{equation}
Also, 
\[ H(G^{-2k}(\tau^{-1}\Omega_-)) = \tau^{2k} H\tau^{-1}(\Omega_-) =
\tau^{2k} G(\Omega_-) = \tau^{2k}\Omega_+ \; .\]

Now take $x$ in the Julia and in $\tau^{-p}(\Omega_-\setminus
\tau^{-1}\overline{\Omega_-})$. Without loss of generality $p\geq 0$
since the case of $x\in \Omega_+$ was already considered. If $x$ is
not in the rescaled image of one of the exceptional domains discussed
in the previous paragraph, then the map associated to $\tilde{H}$ is
just $H_{-p}$. 

If $x$ is  in $\tau^{-p}G^{-2k}(\tau^{-1}(\Omega_-))$, then $H_{-p}$ maps
it into$\tau^{2k-p}\Omega_+$. But $H_{-p}$ is an iterate of
$H$, so it has to map $x$ into the Julia set of $H$ and thus 
$2k-p\leq 0$
or $2k\leq p$. By formula~(\ref{equ:2jp,1}), the associated map is
given by 
\[ \tau^p H_{2k} \circ \cdots \circ H_1 \tau^{-p} \]
which is clearly generated by $\tau^p H_{2k} \tau^{-p} = H_{2k-p}$,
thus by $H_0$ in view of the inequality $2k\leq p$. 

What we now proved is that if $x\in \Omega_-$, then the map associated
to ${\tilde J}$ is an iterate of $H$ on a neighborhood. This is the
same as the map associated to $\Lambda$ unless $x\in \tau^{-1-p}
\Omega_+$ for $p\geq 0$. If that happens, $\Lambda = {\tilde
J}\circ{\tilde J}$ and the map associated to $\tilde J$ is $H_{-p}$ on
a neighborhood of $x$ and therefore maps $x$ into $\tau^{-p}\Omega_-$.  
Then, again the map associated to the second iterate of $\tilde J$ is
generated by $H$ on a neighborhood of $H_{-p}(x)$. 

Proposition~\ref{prop:2jp,1} has been demonstrated.

\section{Drift Estimates}

\subsection{Martingale estimates}
We will be using the following abstract probabilistic
statement. Its stronger form under stronger conditions
can be found in the literature, see the discussion and references in
the Introduction.

Define $\gamma_k(x) = x \chi_{[-k,k]}(x)$ for $k>0$. 
\begin{prop}\label{prop:14za,1}
On a certain probability space $\Omega$ with measure $\mu$ consider an integer-valued stochastic process 
$(Z_n)_{n=0}^{\infty}$. Let ${\cal F}_n$ denote the $\sigma$-algebra
generated by $Z_0,\cdots, Z_n$.  For $n\geq 1$, let $F_n = Z_n -
Z_{n-1}$. Assume that for each $n\geq 1$ we have a decomposition $F_n =
\Delta_n + I_n$, with $\Delta_n$ and $I_n$ both integer-valued. Moreover, assume that
positive constants $K_1, K_2$, $p>1$, exist with which the following estimates
hold for every $n\geq 1$: 
\begin{itemize}
\item
for every $k\in\ZZ, k\neq 0$
\[ K_1^{-1} k^{-2} \leq  P({\Delta_n = k}|{\cal F}_{n-1})(\omega) \leq K_1 k^{-2} \]
for $\mu$-almost all $\omega$, 
\item
for every positive $k$, 
\[ |E(\gamma_k(\Delta_n) |{\cal F}_{n-1})(\omega)| \leq K_2 \] almost surely, 
\item
\[ E(|I_n|^p){\cal F}_{n-1})(\omega) \leq K^p_2 \] almost surely.
\end{itemize}

Then, $\mu$-almost surely $\lim\sup_{n\rightarrow\infty} \frac{Z_n}{n} =
+\infty$ and $\lim\inf_{n\rightarrow\infty} \frac{Z_n}{n} =
-\infty$. 
\end{prop}

Let us define $\log^+(x)$, $x\in \RR$ to be $\log(x)$ if $x>1$ and 
$0$ otherwise. 

\begin{lem}\label{lem:12za,1}
Consider a probability space $P$ with measure $\mu$. Let $\Delta$
and $I$ be integer-valued random variables and $F=\Delta+I$. Assume
that for some $Q', Q''>0$, $p>1$ and every $k\neq 0$:
\begin{itemize}
\item
\[ (Q' k^2)^{-1} \leq    \mu({\Delta=k}) \leq Q' k^{-2} \; ,\]
\item
if $k>0$, then 
\[ |E(\gamma_k(\Delta))| \leq Q'' \; ,\]
\item
\[ E(|I|^p) \leq (Q'')^p\]
\end{itemize}

There exists $Q_0 > 1$ which only depends on $Q',Q'',p$ such that for
every $Q \geq Q_0$ 

\[ E(\log^+ (Q+F)) < \log Q \; .\]
\end{lem}
\begin{proof}
Without loss of generality we can replace $I(x)$ with $\max
(I(x),0)$. i.e. assume that $I(x)$ is a non-negative function. 

Assume $Q>800$ and distinguish sets $X_Q := \{ x\in P :\: Q+\Delta(x) > 1\}$
and $Y_Q := \{ x\in P :\: |\Delta(x)|\leq Q\log Q - Q\}$.

\[ \int_{P\setminus Y_Q} \log^+(Q+F(x))\, d\mu(x) \leq \int_{X_Q \setminus
Y_Q} \log^+ (Q+F(x))\, d\mu(x) + \]
\[ + \int_{P\setminus (Y_Q\cup X_Q)}
\log^+ \left[ I(x)- Q(\log Q -2)  \right]\, d\mu(x) \]
since on the complement of $X_Q \cup Y_Q$ we have $\Delta(x)< -Q\log Q+Q$. 

To estimate that last term, denote $S=\{ x\in P\setminus (Y_Q\cup
X_Q):I(x)\geq Q(\log Q-2) + 1)$. 
Using Jensen's inequality for conditional expectations

\[ \int_{P\setminus (Y_Q\cup X_Q)}
\log^+ \left[ I(x)-Q(\log Q - 2)\right]\, d\mu(x) \leq \int_S \log I(x)\, d\mu(x) =
\] 
\[ = \mu(S) E(\log I(x)| x\in S)  \leq \mu(S) \log E(I(x)|x\in S)\; .\]
Furthermore, 
\[ E(\log I(x) | x\in S) \leq \frac{E(I)}{\mu(S)} \leq
\frac{Q''}{\mu(S)} \]
and 
\[ \mu(S) \log E(I(x)|x\in S) \leq \mu(S) \log Q'' - \mu(S)\log(\mu(S))
\; .\]

Since $Q>800$, we have 
\[ \frac{Q}{2}\log Q \cdot\mu(S) < Q(\log Q-2)\mu(S)  < \int_S I(x)\,
d\mu(x) \leq Q''\]
which implies  $\mu(S) < \frac{2Q''}{Q\log Q}$ and from the previous estimate   
\[  \int_{P\setminus (Y_Q\cup X_Q)} \log^+ \left[ I(x)-Q(\log Q - 2)
  \right]\,
d\mu(x) \leq \] 
\[ \leq \frac {2Q''\log Q''}{Q\log Q} + \frac{2Q''}{Q\log Q}(\log Q +
\log\log Q - \log (2Q'')) \leq \frac{Q'_1}{Q} \]
for an appropriately chosen constant $Q'_1$ which only depends on
$Q''$.

\[\int_{P\setminus Y_Q} \log^+(Q+F(x))\, d\mu(x) \leq \frac{Q'_1}{Q} 
+ \int_{X_Q \setminus Y_Q} \log (Q+\Delta(x))\, d\mu(x) 
+\]
\[ +\int_{X_Q\setminus Y_Q} \frac{I(x)}{Q+\Delta(x)}\, d\mu(x)  \leq 
 \frac{Q'_1}{Q} + \sum_{n> Q\log Q}  Q' \log (Q + n) n^{-2} +
\frac{Q''}{Q} \]
\begin{equation}\label{equ:12za,1}
 \leq \frac{Q'_1}{Q} + \frac{Q''}{Q} + 2Q' \sum_{n\geq Q\log Q} \log (n) n^{-2}  \leq Q_1
 Q^{-1} 
\end{equation}

where the final estimate arises from an explicit integration of the
function $x^{-2}\log x$ and $Q_1$ only depends on $Q'$ and $Q''$.   

Let $\lambda_Q$ denote the affine function tangent to $\log x$ at
$Q$, i.e. $\lambda_Q(x) = \log Q + \frac{x}{Q} -1$.  Then 
\begin{equation}\label{equ:12za,2}
\int_{Y_Q} \log^+(Q+F(x))\, d\mu(x) = 
\end{equation}
\[ = \int_{Y_Q}
\lambda_Q(Q+F(x))\; d\mu(x)  - \int_{Y_Q}
(\lambda_Q-\log^+)(Q+F(x))\, d\mu(x) \; .\]

As to the first term, we estimate
\[ \int_{Y_Q} \lambda_Q(Q+F(x))\; d\mu(x) = \]
\[ = \log(Q)\mu(Y_Q) +
\frac{1}{Q} \int_{Y_Q} F(x) \, d\mu(x) -1  < \log(Q) + (2Q'') Q^{-1} \]

Taking this into account together with
estimates~(\ref{equ:12za,1}) and~(\ref{equ:12za,2}), we get 
\begin{equation}\label{equ:12za,3}
\int_{P} \log^+ (Q + F(x)) \, d\mu(x) - \log Q < \frac{Q_1+
2Q''}{Q} - 
\end{equation}
\[ - \int_{Y_Q}\left(\lambda_Q - \log^{+})(Q + F(x)\right) \, d\mu(x)\; .\]

The rest of the proof will consist in estimating the final negative
term in~(\ref{equ:12za,3}) to show that it goes to $0$ as
$Q\rightarrow\infty$ more slowly than $O(Q^{-1})$ and hence prevails
for sufficiently large $Q$. 
The values of $\lambda_Q(x)$ remain above $\log_+ (x)$ for $x>-Q\log Q+Q$. 
Since $I(x)$ is non-negative and $\Delta(x) \geq -Q\log Q+Q$ on $Y_Q$, 
$(\lambda_Q -\log^+)(Q+F(x))$ is non-negative on $Y_Q$. Choose $Z_Q :=
\{ x\in P :\: -3Q < \Delta(x) < -2Q \}$. Since $Q>800$, $Z_Q \subset
Y_Q$ and 
\begin{equation}\label{equ:23fa,1}
\int_{Y_Q}\left(\lambda_Q - \log^{+})(Q + F(x)\right) \, d\mu(x) \geq 
\int_{Z_Q}\left(\lambda_Q - \log^{+})(Q + F(x)\right) \, d\mu(x)
\end{equation}

For $Q>800$ and $x\in Z_Q$, 
\[ \lambda_Q(Q+\Delta(x)) \geq \log Q - 3 > \frac{\log
Q}{2}\; .\] 

At the same time, for $x\in Z_Q$, 
\[ Q+F(x) = Q+\Delta(x)+I(x) < I(x)-Q\ .\] 

Hence, 

\[ \int_{Z_Q} \left(\lambda_Q - \log^+\right) (Q + F(x))\, d\mu(x) >
\]
\[ > \frac{\log Q}{2}\mu(Z_Q) + \int_{Z_Q} \left[ \frac{I(x)}{Q} -
\log^+(I(x)-Q)\right] \, d\mu(x) \;
.\]

By the hypothesis of the lemma, $\mu(Z_Q) > 2Q_2/Q$ for some
positive $Q_2$ where $Q_2$ depends only on $Q''$ and so 
\[ \int_{Z_Q} \left(\lambda_Q - \log^+\right) (Q + F(x))\, d\mu(x) > \]
\[ > Q_2 \frac{\log
Q}{Q} + \int_{Z_Q} \left[ \frac{I(x)}{Q} - \log^+(I(x)-Q)\right]\, d\mu(x) \; .\]

In the integral term, the integrand is non-negative if $I(x)\leq Q+1$ or
$I(x) \geq Q^2$, keeping in mind that $Q>800$. For other values of
$x$, the lower bound by $-\log Q^2$ holds.
It follows that 
\[  \int_{Z_Q} \left[ \frac{I(x)}{Q} - \log^+(I(x)-Q)\right]\, d\mu(x)
> - \log Q^2 \mu(\{ x :\: Q<I(x)<Q^2\}) \; .\]

Since 

\[ \int_{Q<I(x)<Q^2} I^p(x)\, d\mu(x) > 
 (Q'')^p \mu(\{ x :\: Q<I(x)<Q^2\}\; ,\]
one gets  

\[ \mu(\{ x :\: Q<I(x)<Q^2\}) < \frac{(Q'')^p}{Q^p} \]
Thus, 
\[ \int_{Z_Q} \left(\lambda_Q - \log^+\right) (Q + F(x))\, d\mu(x) > Q_2 \frac{\log
Q}{Q} - 2  \frac{(Q'')^p}{Q^{p-1}} \frac{\log Q}{Q} > \frac{Q_2}{2} \frac{\log
Q}{Q}\]
for $Q\geq Q_0 = (\frac{4 (Q'')^p}{Q_2})^{\frac{1}{p-1}}$.

Hence for $Q\geq Q_0$, in view of~(\ref{equ:23fa,1}), 
the negative term on the right-hand side of
estimate~(\ref{equ:12za,3}) dominates and that proves the assertion of
Lemma~\ref{lem:12za,1}. 
\end{proof}

Lemma~\ref{lem:12za,1} will be used with $Q'=K_1$ and $Q''=K_2$ from 
Proposition~\ref{prop:14za,1}. This defines a constant $Q_0$.

\paragraph{Supermartingale construction.}

Choose $N\geq 0$. Under the hypothesis of
Proposition~\ref{prop:14za,1}, define a stochastic process
$(\zeta_n^{(N)})_{n\geq N}$ as follows. If for some $N\leq k\leq n$,
$Z_n(x)<Q_0$, then pick the smallest such $k$ and set $\zeta_n^{(N)}(x) =
\log^+ Z_k(x)$. Otherwise, let $\zeta_n^{(N)}(x)=\log Z_n(x)$. In other
words, $\zeta_n^{(N)}$ is the process $\log^+ Z_n$ starting at $N$ and
stopped when $Z_n$ first dips below $Q_0$. 

\begin{lem}\label{lem:23fa,1}
For every $N\geq 0$, $\zeta_n^{(N)}$ is a supermartingale with respect
to the filtration $({\cal F})_n$ and converges almost surely to a
finite limit.    
\end{lem}
\begin{proof}
If $\zeta_{n-1}^{(N)} < \log Q_0$, then the process is stopped and its
conditional increment is $0$. Otherwise, if $\zeta_{n-1}^{(N)} = \log
Q \geq \log Q_0$, Lemma~\ref{lem:12za,1} can be applied to the
conditional increments. That, we put $F,\Delta,I$ equal to $F_n,
\Delta_n, I_n$, respectively and the probabilistic space is the set 
$S=\{ \omega :\: \zeta_{n-1}^{(N)}(\omega) = Q\}$ with normalized
measure $\mu$. Then the Lemma says that
$E(\zeta_{n}^{(N)}-\log Q|{\cal F}_{n-1})(\omega) < 0$ almost surely
on $S$. 

Since $\zeta_{n}^{(N)}$ is non-negative by definition, it converges
almost surely by martingale theory.  
\end{proof}   

\paragraph{Proof of Proposition~\ref{prop:14za,1}.}
We will first show that $\lim_{n\rightarrow\infty} Z_n = +\infty$ with
probability $0$. Suppose otherwise. Then there is $N$ such that with
positive probability $Z_n(x) > Q_0$ for all $n\geq N$ and
$\lim_{n\rightarrow\infty} Z_n(x) = +\infty$. Considering
$\zeta_n^{(N)}$ we see that on this set $\zeta_n^{(N)}(x) = \log
Z_n(x)$ for all $x$ and thus diverges to $\infty$ contrary to the
assertion of Lemma~\ref{lem:23fa,1}. 

Now pick an arbitrary $M>0$ and consider the process $\tilde{Z}_n =
Z_n + nM$. It is measurable with respect to the same filtration $({\cal
F})_n$ and evidently satisfies the hypothesis of
Proposition~\ref{prop:14za,1}, since we can just set $\tilde{I}_n = I_n
+ M$ for all $n$. The hypothesis of Proposition~\ref{prop:14za,1} is
satisfied with the same $K_1$ and $K_2:=K_2+M$. 
Hence, the conclusion that
$\lim_{n\rightarrow\infty} \tilde{Z}_n = \infty$ almost nowhere
remains valid.     

But that means $Z_n < -Mn/2$ infinitely often almost surely, and so 
\[ \lim\inf_{n\rightarrow\infty} \frac{Z_n}{n} \leq M/2 \; .\]
Since $M$ was arbitrary, we further conclude that  
\[ \lim\inf_{n\rightarrow\infty} \frac{Z_n}{n} = -\infty \]
almost surely and by considering the process $(-Z_n)$ instead of
$(Z_n)$, we also get that the upper limit of $\frac{Z_n}{n}$ is
$+\infty$ almost surely.

\subsection{The drift function}
Based on Lemma~\ref{lem:22xp,2} we can define combinatorial
displacements for all branches induced by $\tilde{H}$ by simply adding
the displacements for all branches of $\tilde{H}$ that occur in the
composition. It will then remain true that if a branch $\zeta$ of the
induced map has combinatorial displacement $p$, then $\tau^p\zeta$
belongs to the tower. 

\begin{defi}\label{defi:22xp,1}
Given a map ${\cal J}$ induced by $\tilde{H}$, define its {\em drift
function} $\Delta_{\cal J}$ to be equal on the domain of any branch of
$\cal J$ to the combinatorial displacement of that branch. 
\end{defi} 

Define 
\[ \gamma_n(x) := \left\{\begin{array}{ccc} 0 & \mbox{if} &  x\geq n\\
                                    x & \mbox{if} & -n < x < n\\
                                    0 & \mbox{if} & x\leq -n \; .
                          \end{array} \right. \]

Fix one of the four pieces $K_{\pm}, L_{\pm}$ and denote it $P$. The
set $M_P$ consists of all probabilistic measures $\mu$ on $P$ which can be
obtained as $\mu = \zeta_*( Q\lambda )$ where  $\zeta$ is a branch of 
$\tilde{J}^n$, for any $n\geq 1$, which
maps onto $P$, $\lambda$ is the
Lebesgue measure and $Q$ a normalizing constant equal to the
reciprocal of the area of the domain of $\zeta$.

Define the function $\Delta^0$ as follows: $\Delta^0(z) = n$ if $z\in
V_{k,n}$ for $k\neq 0,1$ and $n\in \ZZ$ and $0$ otherwise. Then
$\Delta^0$ coincides with $\Delta_{\tilde{H}}$ except on the ``central
rows'' $V_{k,n}$, $k=0,-1$. The idea of the proposition to follow is
that $\Delta^0$ is a good approximation of the much more complicated
function $\Delta_{\cal J}$ and that $\Delta^0$ has certain helpful
properties.

\begin{prop}\label{prop:22xp,2}
If $P$ is one of $K_{\pm}, L_{\pm}$, then there exist positive $Q_1,
Q_2, Q_3$ so that for every $\mu\in M_P$:
\begin{itemize}
\item
\[ \int_P |\Delta_{\tilde{\cal J}} - \Delta^0|^{\frac{3}{2}}\, d\mu < Q_1 \; ,\] 
\item
for every  $n\neq 0$, 
\[ Q_2^{-1} |n|^{-2} < \mu(\{ x\in P :\: \Delta^0 = n\}) < Q_2 |n|^{-2} \; ,\]
\item
for all $n$ 
\[ |\int_P \gamma_n \circ \Delta^0\, d\mu | \leq Q_3\; .\]
\end{itemize}
\end{prop}

Observe that the first two properties would be enough to prove for the Lebesgue
measure instead of $\mu$, since the densities $\frac{d\mu}{d\lambda}$
bounded for all $\mu\in M_P$ in view of bounded distortion.

The last property deserves attention. Although
$\Delta^0$ is non-integrable in view of the second claim, its
integrals in a certain principal value sense remain bounded. Also,
this one would not be enough to prove for the Lebesgue measure as it
involves cancellations. 

\paragraph{Proof of Proposition~\ref{prop:22xp,2}.}
Let us start with the following general Lemma. 


\begin{lem}\label{lem:24fp,1}
Let $\Phi$ be a holomorphic function defined on a neighborhood of $0$,
with the power series expansion at $0$ in the form
\[ \Phi(z) = z + az^3 + O(|z|^4) \]
with some complex $a\neq 0$. Choose $\tilde h$ to be its Fatou coordinate, so that 
\[ \tilde h\circ\Phi(z) = \tilde h(z)+1 \] for all $z$ in an attracting petal of
$0$. Let $f,g$ be continuous functions defined for $x>r>0$ for some
$r$ such that 
$f(x) > g(x)$ for all $x$ and $1$-periodic.

There exists $K$ so that for any $n>r$ the area of the set 
\[ \tilde h^{-1}  ( \{ x+iy :\: n<x<n+1, g(x) < y < f(x)\} ) \]
is bounded above by $Kn^{-3}$. 
\end{lem}
\begin{proof} 
It is well known (see also the proof of Lemma~\ref{lem:24fp,5})
that the 
\[ |(\tilde h^{-1})'(z)| = \tilde L |z|^{-3/2} + o(|z|^{-3/2})\; .\]
Hence, the preimage by $\tilde h$ of any square \[\{ x+iy :\: n<x<n+1,
c<y<c+1\}\] for $n$ large has area bounded by $K_1 n^{-3}$ and the
hypotheses of continuity and 1-periodicity for $f,g$, any region in
the form $ \{ x+iy :\: n<x<n+1, g(x) < y < f(x) \} $ is contained in 
the union of $m$ such squares with $m$ independent of $n$. 
\end{proof}

Observe that under $\tilde h^{-1}$, the graphs of $f$ and $g$ are mapped to
curves invariant under $\Phi$ and tangent to the attracting direction
of $\Phi$ at $0$ and, conversely, any two such curves give rise to
functions, $f,g$ which satisfy the hypotheses of
Lemma~\ref{lem:24fp,1}. 

\begin{lem}\label{lem:24fp,2}
Function $\Delta_{\tilde{C}}^{\frac{3}{2}}$ is integrable with respect to the
Lebesgue measure on $W_{0,0}$. 
\end{lem}
\begin{proof}
By the definition of $\tilde{C}$, $\Delta_{\tilde{C}}(z)$ is equal to
the number of iterates of $\tilde{H} = \tau^{-1} G \tau$ needed to map
$z$ outside of $W_{0,0} \cup W_{-1,0}$, which is bounded above by
twice the number of iterates of $\tilde{H}^2$ needed to map $z$
outside of $W_0,0$. $\tilde{H}$ has a degenerate neutral fixed point
at $\tau^{-1}x_0$ in a neighborhood of the fixed point $W_{0,0}$ is
just the complement of $\tau^{-1}(\Omega_-\cup\Omega_+)$ whose
boundary is mapped invariant under $G^2$ if  neighborhood is small
enough. Once
$z$ leaves that fixed neighborhood of the fixed point, it will leave 
$W_{0,0}$ after a bounded number of further iterations. If we apply
Lemma~\ref{lem:24fp,1} to $\Phi := \tilde{H}^{-1}$ we get that the
measure of the set $S_n$ of points $z$ which stay in the neighborhood for
exactly $n$ iterates of $\tilde{H}^2$ is bounded by $Kn^{-3}$. Since
$\Delta_{\tilde{C}}$ on $S_n$ is bounded by $n$ plus a $Q$, the
the integral of $\Delta_{\tilde{C}}^{\frac{3}{2}}$ over $W_{0,0}$ is bounded by 
\[ 2C |W_{0,0}| + 2K \sum n^{-\frac{3}{2}} < \infty \; .\]
\end{proof}

\begin{lem}\label{lem:24fp,3}
For a certain $Q_4$ 
\[ \int_P |\Delta_{\tilde {\cal J}} - \Delta_{\tilde{H}}|^{\frac{3}{2}} \, d\lambda < Q_4 \; .\] 
\end{lem}
\begin{proof}
Since $\tilde{\cal J}$ is either $\tilde{H}$, or $\tilde{C} \circ \tilde{H}$
if $\tilde{H}$ maps into $W_{0,0}\cup W_{-1,0}$,  
\[ \Delta_{\tilde{\cal J}} = \Delta_{\tilde{H}} + \Delta_{\tilde{C}} \]
where we put $\Delta_{\tilde{C}}$ equal to $0$ outside the domain of
$\tilde{C}$.

\[ \int_P |\Delta_{\tilde{\cal J}} - \Delta_{\tilde{H}}|^{\frac{3}{2}}\, d\lambda =  
\int_P |\Delta_{\tilde{C}}(\tilde{H}(z))|^{\frac{3}{2}}\, d\lambda(z)  = \]
\[ =\int_{(W_{0,0}\cup W_{0,-1}) \cap P} |\Delta_{\tilde{C}}(w)|^{\frac{3}{2}}
 |(H^{-1})'(w)|^3 \, d\lambda(w) \; .\] 

The derivative of $\tilde{H}^{-1}$ is bounded on the central
pieces, since $\tilde{H}$ is univalent and maps onto a neighborhood of
their closure. Thus, $\Delta_{\tilde{C}}$ is multiplied by a bounded
factor and hence, in view of
Lemma~\ref{lem:24fp,2}, the integral is finite. 
\end{proof}

\begin{lem}\label{lem:24fp,4} 
There is $Q_5$ so that 
\[ \int_P |\Delta_{\tilde{H}} - \Delta^0|^{\frac{3}{2}}\, d\lambda < Q_5 \; .\]
\end{lem}
\begin{proof}
Let $\chi_0$ be the characteristic function of the ``central rows'',
i.e. the union of pieces $V_{k,n}$, $k=0,-1$, $n\in \ZZ$.   

Clearly, 
\[ \Delta_{\tilde{H}} - \Delta^0 = \chi_0\Delta_{\tilde{H}}\; .\] 

Recall that on $W_{0,k}, W_{-1,k}$, the combinatorial displacement is
just $k$. When $k$ is positive and even, then $G^{k}$ is used to map 
$V_{0,k} \setminus W_{0,k}$ onto $\tau^{-p} B_+$ and the combinatorial
displacement is $k-p$.  The dynamics on $V_{-1,k}$ is the mirror image
of this. On $V_{0,k}$,  
$\Delta_{\tilde{H}}$ is $k-p(z)$ where $p(z)$ is zero unless $k$ is
positive and even, in which case it given by the condition
$z \in G^{-k} \tau^{-p(z)} B_+$. Now $G^{k-1}$ maps $W_{0,k}$ with
bounded distortion into a neighborhood of $\tau x_0$, which is the
critical of $G$. Since $G^k$ is univalent on $W_{0,k}$, it follows  
that
the area of the set of $z\in V_{0,k}$ such that $p(z) = p$ is bounded by 
$Q_1 |V_{0,k}| |p|^{-3}$. It follows that the integral of
$|\Delta_{\tilde{H}}|^{\frac{3}{2}}$ over $V_{0,k}$ is bounded by
$|V_{0,k}|(|k|^{\frac{3}{2}}+10 Q_1)$. By Lemma~\ref{lem:24fp,1},
$|V_{0,k}| \leq K k^{-3}$ so that integral of $|\Delta_{\tilde{H}}|^{\frac{3}{2}}$
over the union of all pieces $V_{0,k}, k\in \ZZ$ is finite. The same
reasoning is applied to pieces $V_{-1,k}$. 
\end{proof}

From Lemmas~\ref{lem:24fp,3} and~\ref{lem:24fp,4}, we derive the first
claim of Proposition~\ref{prop:22xp,2}.  

We will now deal with the remaining two claims which are only
concerned with the function $\Delta^0$.  

Start by defining sets $V_n = \bigcup_{k\not=0,-1} V_{k,n}$.



\begin{lem}\label{lem:24fp,5}     
There exist $C,K_1>0$, such that for all $n$

\[ \lambda(V_n) - C |n|^{-2} \leq K_1 |n|^{-5/2} \; .\]
\end{lem}
\begin{proof}
Consider the map $h^{-1}$ from the slit plane 
$C_h$ onto $\Omega_-$ as described by item (4) of Theorem~\ref{prop}.
The measure of $V_n$ is equal to the integral of $|(h^{-1})'|^2$ over the set 
$S_{n}\cup \bar S_n$, where 
$\bar S_n=\{z: \bar z\in S_n\}$ is the mirrow symmetric to $S_n$ set,
and $S_n$ is a set in the upper half plane $\HH^+$, which is a
``half-strip''
bounded by the horizontal line $\Im z=\pi$ and two transversal curves
$\log(\partial \Omega_-)+(n-1)\log\tau$,
$\log(\partial \Omega_-)+n\log\tau$.


To estimate the integral as $n\to \pm\infty$ we use the parabolic 
fixed point theory applied to the map
$G^2(z)=z - A(z-x_0)^3 + \cdots$, where $A>0$.
The map $h_a:=(-2\log\tau)^{-1}h$ is an attracting Fatou
coordinate of the neutral fixed point $x_0$ of $G^2$:
$h_a\circ G^2(z)=\sigma\circ h_a(z)$,
for $z\in \Omega$ where $\sigma(w)=w+1$ is the shift. 
According to the general theory,

\[h_a(z) = \phi_a(L(z-x_0)^{-2})\] 
where
$L=(2A)^{-1/2}$ 
and
$\phi_a(w)=w+O(|w|^{1/2})$, as $|w|$ tends to $\infty$
in some sector $\Sigma_a=\{w: \Re w>c-\Im w\}$, $c>0$.
Similarly, there exists a repelling Fatou coordinate
$h_r$, such that
$h_r\circ G^2(z)=\sigma\circ h_r(z)$
for $z\in G^{-2}(\HH^{\pm})$, and

\[h_r(z) = \phi_r(L(z-x_0)^{-2})\] 
with the same constant $L$ as for $h_a$, and
$\phi_r(w)=w+O(|w|^{1/2})$, as $|w|$ tends to $\infty$
in a sector $\Sigma_r=\{w: \Re w<-c+\Im w\}$.

We have:

\[|(h_{a}^{-1})'(w)|^2=L/4 |w|^{-3}(1+O(|w|^{-1/2}))\]
as $|w|\to +\infty$ in $\Sigma_a$, and similarly

\[|(h_{r}^{-1})'(w)|^2=L/4 |w|^{-3}(1+O(|w|^{-1/2}))\]
as $|w|\to +\infty$ in $\Sigma_r$.

Note that the picture is mirrow symmetric w.r.t. the real axis.
In particular, $h_a(\bar z)=\overline{h_a(z)}$ etc.

Since we apply $h^{-1}(w)$ as $\Re w\to \pm\infty$,
introduce a pasting map (called also a horn map) $\Psi=h_r\circ h_a^{-1}$.
The map $\Psi$ has an analytic extension from $\Sigma_a\cap \Sigma_r$ to the
upper and lower half planes, it commutes with the shift $\sigma$,
and $\Psi(w)=w+O(|w|^{1/2})$ as $\Im w\to \infty$. It follows, that
\begin{equation}\label{assy}
\Psi(w)=w+v_{\pm}+O(\exp(-\pi |\Im w|))
\end{equation}
uniformly in half-planes compactly contained in 
$\HH^{\pm}$, where $v_{\pm}$ are two complex conjugated vectors.  

By the symmetry, the area $|V_n|$ of $V_n$ 
is twice the area of 
\[ h^{-1}(S_n)=h_{a}^{-1}(S_n/(-2\log\tau))\; .\]
Notice that $S_{n}=S_0+n\log\tau=\cup_{m=0}^{+\infty} (P+(n\log \tau + i\pi
m))$ where
$P$ is a ``rectangle'' bounded by the curves
$\Im z=\pi$, $\Im z=2\pi$ and 
$\log(\partial \Omega)-\log\tau$,
$\log(\partial \Omega)$. 
We will denote $\hat S_n=(-2\log\tau)^{-1}S_n$ etc.
The sets $\hat S_n$, $\hat S_0$, $\hat P$ switch the half planes,
i.e. lie in $\HH^-$.
Thus,

\[|V_n|=2\int\int_{\hat S_n} |(h_{a}^{-1})'(w)|^2 d\sigma_w
=2 \int\int_{\hat P}\sum_{m=0}^\infty
|(h_{a}^{-1})'(t-\frac{n}{2}-\frac{i\pi m}{2\log\tau})|^2 d\sigma_t \]
where $d\sigma_z$ denotes the area element of a complex variable $z$.

First, let $n\to -\infty$, so that 
$\Re (t-\frac{n}{2}-\frac{i\pi m}{2\log\tau})\to +\infty$.
By the asymptotics of $(h_a^{-1})'(w)$ in $\Sigma_a$, 


\[|V_n|=\frac{2L}{4}\, \times \]
\[ \times \int\int_{\hat P}
\sum_{m=0}^\infty\left[ |t+\frac{|n|}{2}-\frac{i\pi m}{2\log\tau}|^{-3} + 
O(|t+\frac{|n|}{2}+\frac{i\pi m}{-2\log\tau}|^{-7/2})\right]\, d\sigma_t.\]

Since $t$ belongs to a bounded domain $\hat P$,
one can replace the sums by corresponding integrals
and arrive at the following asymptotic formula:

\[|V_n|=\frac{4L |\hat P| I \log\tau}{\pi}\frac{1}{|n|^{2}} + \Delta_1(n),\]
where $I=\int_{0}^\infty dx/(1+x^2)^{3/2}$, 
and
$|\hat P|$ is the area of the
bounded domain $\hat P$, and $|\Delta_1(n)|<K_1|n|^{-5/2}$, for some
$K_1$ and all negative $n$.

As for $n$ positive, we can write
(assuming for definiteness that $n$ is even)

\[h_a^{-1}(\hat S_n)=h_{a}^{-1}(\hat S_0-n/2)=
h_{a}^{-1}\circ \sigma^{-n/2}(\hat S_0)= \]
\[ =G^{-n}\circ h_{a}^{-1}(\hat S_0)=
h_r^{-1}\circ \sigma^{-n/2}\circ \Psi(\hat S_0)=
h_r^{-1}(\Psi(\hat S_0)-n/2).\]
As $n\to +\infty$,
using the asymptotics for $(h_r^{-1})'(w)$ in $\Sigma_r$
and (\ref{assy}) for $\Psi$,

\[|V_n|=\frac{2L}{4}\, \times \]
\[ \times \int\int_{\hat P}
\sum_{m=0}^\infty \{ |t-\frac{n}{2}-\frac{i\pi m}{2\log\tau} +v_-
+O(\exp(-\frac{m \pi}{\log \tau}))|^{-3}  + \] 
\[ + O(|t-\frac{n}{2}-\frac{i m \pi}{2\log\tau}+v_-
+O(\exp(-\frac{m \pi}{\log \tau})|^{-7/2})\} 
\{ 1+O(\exp(-\frac{m \pi}{\log\tau}))\} d\sigma_t.\]

One rewrites it as


\[|V_n|=\frac{L}{2}\, \times \]
\[ \times \int\int_{\hat P}
\sum_{m=0}^\infty \{ |t-\frac{n}{2}-\frac{i m \pi}{2\log\tau}+v_-|^{-3}  +  
 O(|t-\frac{n}{2}-\frac{i m \pi}{\log\tau})|^{-7/2})\} 
\{ 1+O(\exp(-\frac{m \pi}{\log\tau}))\} d\sigma_t \]
\[=
\frac{L}{2}\, \times \] 
\[ \times \int\int_{\hat P}
\sum_{m=0}^\infty \{ [ |t-\frac{n}{2}-\frac{i m \pi}{\log\tau}+v_-|^{-3} ] 
[ 1+O(\exp(-\frac{m \pi}{\log\tau}) ] 
 + O(|t-\frac{n}{2}-\frac{i m \pi}{2\log\tau})|^{-7/2})\} d\sigma_t \]

Now we use 
the invariance of the Lebesgue measure under shifts and get the same
asymptotic formula as for $n\to -\infty$.

\end{proof}
 
Now Lemma~\ref{lem:24fp,5} implies the second claim of
Proposition~\ref{prop:22xp,2}. 

To address that last claim, first define 
\[ c_n(\mu) = \frac{\mu(V_n)}{\lambda(V_n)} \; .\]

Then 
\begin{equation}\label{equ:25fp,1}
\int_P \gamma_N \circ \Delta^0 d\mu = \sum_{n=-N}^N  n c_n(\mu)
\lambda(V_n) \; .
\end{equation}

To uniformly bound this quantity, we will need certain properties of
coefficients $c_n(\mu)$ for $\mu \in M_P$. As the result of bounded
distortion, $|\log c_n(\mu)|$ can be bounded independently of $\mu$,
but need stronger properties. 

\begin{lem}\label{lem:25fp,1}
Let $n$ be any integer with $|n| > 1$. Then there exists a constant
$Q$ so that for any $n$ and $\mu\in M_P$
\begin{itemize}
\item
\[ |c_n(\mu) - c_{n+1}(\mu)| < Q |n|^{-3/2} \]
\item
\[ |c_n(\mu) - c_{-n}(\mu)| < Q |n|^{-1/2} \; .\]
\end{itemize}
\end{lem}
\begin{proof}
The basic fact will use, which follows from
Proposition~\ref{prop:22xp,1}, is that functions $\log
\frac{d\mu}{d\lambda}(z)$ are bounded and Lipschitz-continuous,
uniformly for all $\mu\in M_P$. 

For any $k>0$ the set $V_{k,n}\cup V_{k_{n+1}}$ has diameter bounded
by $Q_1 |n|^{-3/2}$. This follows since the derivative of the Fatou
coordinate $h^{-1}(w)$ is asymptotically $|w|^{-3/2}$. By the uniform
Lipschitz property $\log \frac{d\mu}{d\lambda}$ differs by no more
than $O(|n|^{-3/2})$ between any tho points of this set and hence 

\[      (1 - Q_2 |n|^{-3/2} ) \frac{
\mu(V_{k,n+1})}{\lambda(V_{k,n+1})} \leq
\frac{\mu(V_{k,n})}{\lambda(V_{k,n})} \leq (1 + Q_2 |n|^{-3/2})
\frac{\mu(V_{k,n+1})}{\lambda(V_{k,n+1})} \; .\]

Since $c_n(\mu), c_{n+1}(\mu)$ are just averages of these quantities
for various $k$, 

\[ (1-Q_2 |n^{-3/2}|) \leq \frac{c_{n+1}(\mu)}{c_n(\mu)} \leq (1+Q_2
|n^{-3/2}|) \; .\]

Since $c_n(\mu)$ are uniformly bounded above, the first claim
follows. 

To see the second claim, observe that $V_{n}$ and $V_{-n}$ are in
a disk centered at the fixed point with radius $O(|n|^{-1/2})$. This
follows again from the asymptotics $|w|^{-\frac{1}{2}}$ for the Fatou
coordinate $h^{-1}(w)$. 
The uniform Lipschitz estimate then says that
\[ |\frac{d\mu}{d\lambda}(z_1) - \frac{d\mu}{d\lambda}(z_2)| \leq Q_3
|n|^{-\frac{1}{2}} \]
if $z_1 \in V_n$ and $z_2\in V_{-n}$. Since $c_n$ can be bounded above
and below by the extrema of $\frac{d\mu}{d\lambda}(z_1)$ for $z_1\in
V_n$ and $c_{-n}$ can be expressed in an analogous fashion, the
second claim follows.  

\end{proof}

Let us now denote 
\[ B_N = \sum_{n=1}^N n \lambda(V_n) \]
for $N>0$, $B_N = \sum_{n=-N}^{-1} n \lambda(V_n)$ for $N<0$ and
$B_0=0$. 

Applying Abel's transformation to the series in
Equation~\ref{equ:25fp,1}
\[  \int_P \gamma_N \circ \Delta^0 d\mu = \] 
\[ = \sum_{n=1}^{N-1}
B_n(c_{n+1}(\mu) - c_n(\mu)) + 
\sum_{n=-N+1}^{-1} B_n (c_{n-1}(\mu)-c_n(\mu)) +
B_N c_N(\mu) + B_{-N} c_{-N}(\mu) \; .\]  

The first sum can be bounded by 
\[ Q_1 \sum_{n=1}^{N-1} |B_n| n^{-3/2} \]
by Lemma~\ref{lem:25fp,1}. Since $|B_n| < Q_2 \log n$ by
Lemma~\ref{lem:24fp,5}, the sum is uniformly bounded for all $N$ and
$\mu$. The second sum is dealt with in the same way. 

Then 
\[ c_{-N}(\mu) B_{-N} + c_N(\mu) B_N = (c_{-N}(\mu)-c_N(\mu)) B_{-N} +
c_N(\mu)(B_{-N}+B_{N}) \; .\]

By Lemma~\ref{lem:25fp,1}, $(c_{-N}(\mu)-c_N(\mu)) B_{-N}$  goes to $0$ with
$N$. At the same time, $B_{-N}+B_{N}$ are bounded independently of $N$
by Lemma~\ref{lem:24fp,5}, since the leading terms $C |n|^{-2}$ in
$\lambda(V_n)$ give rise to exactly canceling contributions and the 
$O(|n|^{-5/2})$ corrections after multiplying by $n$ result in
convergent series. 

This ends the proof of Proposition~\ref{prop:22xp,2}.

\section{Main results: Proofs}
\paragraph{The level process.}
For $x\in K_+$ and $n>0$ let us define $Z_n$ to be the combinatorial
displacement of the branch of ${\tilde {\cal J}}^n$ whose domain contains
$x$. For Lebesgue-a.e. $x$, $Z_n$ are thus defined for all positive
$n$. We may set $Z_0$ to be $0$ everywhere. If ${\tilde {\cal J}}^n$ maps $x$
into a piece $P$ (where $P$ maybe any of the four pieces $K_{\pm},
L_{\pm}$), then clearly $Z_{n+1} = Z_{n} + \Delta_{\tilde J}({\tilde {\cal J}}^n(x))$. 
The sequence $(Z_n)_{n\geq 0}$ may be viewed as a stochastic process
on a probabilistic space $K_+$ with probability given by the Lebesgue
measure on $K_+$ normalized to total mass $1$.  

To this process we can apply Proposition~\ref{prop:14za,1}, because
its hypotheses are satisfied in view of
Proposition~\ref{prop:22xp,2}. 

\paragraph{The combinatorial displacements for the iterates of
$\Lambda$.}

Recall map $\Lambda$ which is equal to ${\tilde J}$ or ${\tilde J}^2$ on
various pieces of its domain. At almost every point $z$ of $K_+$, we
have a sequence $n_m$ where $\Lambda^m = {\tilde J}^{n_m}$ on a
neighborhood of $z$. In particular, the combinatorial displacement of
$\Lambda^m$ is $Z_{n_m}$. Also, $n_{m+1} - n_{m} \leq 2$. 

\begin{prop}\label{prop:2jp,2}
For almost every $x\in K_+$ both $\lim\inf_{m\rightarrow\infty} Z_{n_m}(x) =
-\infty$ as well as $\lim\sup_{m\rightarrow\infty} Z_{n_m}(x) = +\infty$
hold true. 
\end{prop}

Suppose this is not the case and the first statement fails. Then for a
set $S$ of positive measure $Z_{n_m}(x) \geq M$ for all $m$ and $x\in
S$. Let $x_0$ be a density point of $S$ and, by
Proposition~\ref{prop:14za,1}, 
$\lim\inf_{n\rightarrow\infty} Z_{n}(x) = -\infty$. Choose $n$ so that 
$Z_n(x) < M$. Let $U_n$ be the domain of the branch of ${\tilde J}^n$
which contains $x_0$. By the bounded distortion of ${\tilde J}$, 
$U_n$ for all such $n$ form a basis of
neighborhoods of $x_0$ such that $|U_n| \geq \kappa (\diam U_n)^2$ for
a constant $\kappa > 0$. By the bounded distortion of $\tilde J$, each
$U_n$ contains a fixed proportion of points $x$ for which $Z_{n+1}(x)
< Z_n(x)$. But for all such $x$ either $n$ or $n+1$ is in the
subsequence $n_m$, so none of them belongs to $S$ and $x_0$ is not a
density point. 

$\lim\sup_{m\rightarrow\infty} Z_{n_m}(x) = +\infty$ is proved by
contradiction in the same way.

\paragraph{Theorem~\ref{A} and the symmetry of the tower}
Recall that $H$ is a limiting map introduced in Theorem~\ref{prop}.

Here we prove a statement which is stronger than Theorem~\ref{A}:

\begin{theo}\label{rec}
There is a map $\Phi$ defined on a countable union of disjoint 
open topological disks whose complement in $\CC$ has measure $0$, 
and such that on each connected component of its domain $\Phi$ belongs
to the tower dynamics of $H$, with the following property:

\begin{itemize}
\item almost every point in the plane
visits any neighborhood of zero and infinity under the iterates of
$\Phi$, 
\item
for any point $x$ of the Julia set of $H$ which is in the domain of
$\Phi^p$, $p>0$, $\Phi^p$ is an iterate of $H$ on a neighborhood of
$x$. 
\end{itemize}
\end{theo}

{\bf Remark.} It seems to be natural to call the dynamics of $H$ with
such properties {\it metrically symmetric}.

Map $\Phi$ is defined to be associated, in the sense of
Definition~\ref{defi:21ga,1}, to the induced map $\Lambda$
introduced by Proposition~\ref{prop:2jp,1}.

Proposition~\ref{prop:14za,1} asserts that for almost every point its
combinatorial displacements vary from $-\infty$ to $+\infty$. Recalling
Lemma~\ref{lem:22xp,2}, for almost every point $z$ there is a sequence of
iterates in the maximal tower which map $z$ into $\tau^{k_n} P_{k_n}$
where  $k_n \rightarrow +\infty$ and each $P_{k_n}$ is one of the four
pieces $K_{\pm}, L_{\pm}$. Since all $P_{k_n}$ are contained in a
fixed ring centered at $0$, that means images of $z$ under those
iterates tend to $\infty$. But similarly there is a sequence $l_N
\rightarrow -\infty$ with the same property and images of $z$ under
those iterates tend to $0$.

\paragraph{Finite order Feigenbaum maps: Corollary~\ref{finite}}

We use mainly Theorem~\ref{prop}, see also~\cite{LSw1}.
The Julia set $J(H)$ of $H$ is a compact set.
Fix a neighborhood $V$ of $J(H)$.
To show that the area $|J(H^{(\ell)})|$ tends to zero,
it is enough to show that $J(H^{(\ell)})\subset V$
for all $\ell$ large enough.
To this end, for any point $w$ outside of $V$ there is 
a minimal $j\ge 0$, such that $H^j(w)$ is outside of 
the closure of $\Omega$. Since $H^{(\ell)}$ converges to $H$
uniformly on compact sets in $\Omega$, we have that
also $(H^{(\ell)})^j(w)$ is outside of the closure
of $\Omega$ as well. On the other hand, for every $\ell$,
there is a maximal polynomial-like extension
of $H^{(\ell)})$ to a domain $\Omega_\ell$ onto a slit 
complex plane~\cite{Epstein}.
The boundary of $\Omega_\ell$ is invariant
under $G_\ell^{-1}$, where $G_\ell=H^{(\ell)}\circ \tau_\ell^{-1}$.
Then $G_\ell^{-1}$ converge to $G^{-1}$ in $\HH^{\pm}$ uniformly
on compacts. It follows, that the boundaries of
$\Omega^{(\ell)}$  converge uniformly to the boundary of $\Omega$.
Therefore, $(H^{(\ell)})^j(w)$ is outside of
$\Omega^{(\ell)}$, for $\ell$ large enough, i.e. $w$
is not in the Julia set of $J(H^{(\ell)})$.

This proves Corollary~\ref{finite}. However, on the question of
whether maps of finite order have Julia sets of zero measure, our
method sheds little light, since it is based on the infinite variance of
the drift function, which does not hold in any finite order case.

\end{document}